\documentclass{amsart}
\usepackage{tikz,tikz-cd}
\usepackage{url}
\usepackage{graphicx}
\usepackage{graphicx}
\usepackage{color}
\usepackage{times}
\usepackage{cite}
\usepackage{enumerate,latexsym}
\usepackage{latexsym}
\usepackage{amsmath,amssymb}
\usepackage{graphicx}
\usepackage{amsthm}
\usepackage{verbatim}
\newtheorem{thm}{Theorem}[section]
\newtheorem*{theorem*}{Theorem}
\newtheorem*{acknowledgement*}{Acknowledgement}
\newtheorem{cor}[thm]{Corollary}

\newtheorem{lem}[thm]{Lemma}
\newtheorem{prop}[thm]{Proposition}
\theoremstyle{definition}

\theoremstyle{remark}
\newtheorem{rem}[thm]{Remark}

\newtheorem{conj}[thm]{Conjecture}
\numberwithin{equation}{section}

\newcommand{\set}[1]{\left\{#1\right\}}
\newcommand{\Real}{\mathbb R}

\newcommand{\dist}[0]{\mathrm{dist}}

\title{Colding Minicozzi Entropy in Hyperbolic Space}
\author{Jacob Bernstein}
\address{Department of Mathematics, Johns Hopkins University, 3400 N. Charles Street, Baltimore, MD 21218}
\email{bernstein@math.jhu.edu}
\thanks{The author was partially supported by the NSF Grant  DMS-1609340 and DMS-1904674. }

\begin{document}

\begin{abstract}
	This note introduces a notion of entropy for submanifolds of hyperbolic space analogous to the one introduced by Colding and Minicozzi for submanifolds of Euclidean space. Several properties are proved for this quantity including monotonicity along mean curvature flow in low dimensions and a connection with the conformal volume.
\end{abstract}
\maketitle
\section{Introduction}
In \cite{Coldinga}, Colding and Minicozzi introduced the following functional on the space of $n$-dimensional submanifolds $\Sigma\subset \Real^{n+k}$,
$$
\lambda[\Sigma]=\sup_{\mathbf{x}_0\in \Real^{n+k}, \tau>0} (4\pi \tau)^{-\frac{n}{2}}\int_{\Sigma} {e^{-\frac{|\mathbf{x}(p)-\mathbf{x}_0|^2}{4\tau}}} dVol_{\Sigma}(p).
$$
Here, $(4\pi t)^{-\frac{n}{2}} {e^{-\frac{|\mathbf{x}|^2}{4t}}}$ is the heat kernel of $\Real^n$.  Observe that in the definition of $\lambda[\Sigma]$ this function is extended to all of $\Real^{n+k}$ in the obvious way.
Colding and Minicozzi called this quantity the entropy of $\Sigma$ and observed that, by Huisken's monotonicity formula \cite{HuiskenMon}, it is monotone along reasonable mean curvature flows.  That is, if $\set{\Sigma_t}_{t\in [0,T)}$ is a family of complete submanifolds of polynomial volume growth that satisfy
$$
\frac{d\mathbf{x}}{dt}=\mathbf{H}_{\Sigma_t},
$$
then $\lambda[\Sigma_t]$ is monotone non-increasing in $t$.  Here $\mathbf{H}_{\Sigma_t}=\Delta_{\Sigma_t} \mathbf{x}=-H_{\Sigma_t} \mathbf{n}_{\Sigma_t}$ is the mean curvature vector of $\Sigma_t$ and a submanifold, $\Sigma$, has polynomial volume growth if $Vol_{\Real^{n+k}}(\Sigma\cap B_R)\leq M (1+R)^M$ for some $M>0$.  This notion of entropy has been extensively studied in recent years.  See, for instance, \cite{BWIsotopy, BernsteinWang1, JZhu, CIMW, BernsteinWang2, BernsteinWang3, BWHausdorff, SWang, KetoverZhou}.

In this short note, we introduce an analogous quantity for submanifolds of hyperbolic space, $\mathbb{H}^n$, and show that, at least in low dimensions, it is monotone along mean curvature flow. To that end, let $H_n(t,p; t_0,p_0)$
be the heat kernel on $\mathbb{H}^n$ with singularity at $p=p_0$ at time $t=t_0$. 
That is, $H_n$ is the unique positive solution to
$$
\left\{\begin{array}{cc} \left( \frac{\partial}{\partial t}-\Delta_{\mathbb{H}^n}\right)H_n=0 & t>0\\
 \lim_{t\downarrow t_0} H_n =\delta_{p_0}.\end{array}\right.
$$
It follows from the symmetries of $\mathbb{H}^n$ that 
$$
H_n(t, p;t_0, p_0)=K_n(t-t_0,\dist_{\mathbb{H}^n}(p,p_0))
$$
where $K_n(t,\rho)$ is a positive function on $(0,\infty)\times (0,\infty)$ and $\dist_{\mathbb{H}^n}(p,p_0)$ is the hyperbolic distance between $p$ and $p_0$. 
For $(t,p)\in (-\infty, t_0)\times \mathbb{H}^{n+k}$ let 
$$
\Phi_n^{t_0,p_0}(t,p)=K_n(t_0-t,\dist_{\mathbb{H}^{n+k}}(p,p_0)).
$$
In particular,  $\Phi_n^{t_0,p_0}$ restricts to the backwards heat kernel that becomes singular at $(t_0,p_0)$ on any totally geodesic $\mathbb{H}^n\subset \mathbb{H}^{n+k}$ that goes through $p_0$. 

In analogy with the Euclidean setting, for any $n$-dimensional submanifold $\Sigma\subset \mathbb{H}^{n+k}$, define the \emph{hyperbolic entropy} of $\Sigma$ to be
$$
\lambda_{\mathbb{H}}[\Sigma]=\sup_{p_0\in \mathbb{H}^{n+k}, \tau>0} \int_{\Sigma} \Phi_n^{0,p_0}(-\tau,p) dVol_{\Sigma}(p)
$$
We remark that for any submanifold $\Sigma\subset \mathbb{H}^{n+k}$, $\lambda_{\mathbb{H}}[\Sigma]\geq 1$ and if $\mathbb{H}^n\subset \mathbb{H}^{n+k}$ is a totally geodesic copy of the $n$-dimensional hyperbolic plane, then $\lambda_{\mathbb{H}}[\mathbb{H}^n]=1$ -- see Lemma \ref{HyperplaneEntropyLem}. 

The main result of this paper is that, at least in sufficiently low dimensions, $\lambda_{\mathbb{H}}$ is monotone non-increasing along reasonable mean curvature flows -- i.e., flows of submanifolds of exponential volume growth.   Here a submanifold $\Sigma \subset \mathbb{H}^{n+k}$ has exponential volume growth if
$$
Vol_{\mathbb{H}^{n+k}}\left(\Sigma\cap B_{R}^{\mathbb{H}^{n+k}}(p_0)\right)\leq M e^{M R}
$$
for some $M>0$ and $p_0\in \mathbb{H}^{n+k}$.  In fact, at least in low dimensions, there is an analog of Huisken monotonicity.
\begin{thm}\label{MainThm}
	There is an $N\in [4, \infty]$ so that if $n<N$ and $\set{\Sigma_t}_{t\in [0,T)}$ is a smooth mean curvature flow of $n$-dimensional complete submanifolds $\Sigma_t\subset \mathbb{H}^{n+k}$ that have exponential volume growth, then, for $t_0\in (0,T]$, $p_0\in \mathbb{H}^{n+k}$ and $t\in (0,t_0)$,
	\begin{align*}
	\frac{d}{dt} \int_{\Sigma_t} &\Phi_n^{t_0,p_0}(t, p) dVol_{\Sigma_t}(p) \leq 0
	\end{align*}
   and one has strict inequality unless $\Sigma_t$ is a minimal submanifold that is a cone over $p_0$.
    
    As a consequence, $\lambda_{\mathbb{H}}[\Sigma_t]$ is non-increasing in $t$.
\end{thm}
Here a submanifold, $\Sigma$, is a cone over a point $p_0$, if,  for every $p\in \Sigma$, the minimizing geodesic connecting $p_0$ to $p$ in $\mathbb{H}^{n+k}$ is contained in $\Sigma$.  One readily checks that the only smooth $n$-dimensional cone over $p_0$ is a totally geodesic copy of $\mathbb{H}^n$ containing $p_0$.

\begin{rem}
	In fact,  monotonicity holds provided the heat kernel of hyperbolic space, $H_n$, satisfies the following convexity property:
	$$
	 \nabla^2_{\mathbb{H}^n} \log H_n(t,p; t_0, p_0) - \coth(\rho(p)) \partial_\rho \log H_n(t,p; t_0, p_0) g_{\mathbb{H}^n}\geq 0
	$$
	for all $p\in \mathbb{H}^n$ and $t>t_0$.  Here $\rho(p)=\dist_{\mathbb{H}^n}(p,p_0)$ is the radial distance from $p_0$.  This is equivalent to
	$$
	\partial^2_\rho \log K_n -\coth(\rho)\partial_\rho \log K_n \geq 0.
	$$
	As the $K_n$ can be explicitly computed, one could, in principle, check this condition for any $n$ and it seems reasonable to conjecture that the convexity, and hence Theorem \ref{MainThm}, holds for $N=\infty$.  However, the formulas become unwieldy for $n\geq 4$ and so we do not pursue this. See \cite{daviesHeatKernelBounds1988} for an extensive study of properties of $H_n$ and $K_n$.
\end{rem}
	An immediate consequence of this result, the classification of low entropy self-shrinkers in $\Real^3$ from \cite{BernsteinWang2} and the fact that the mean curvature flow of a closed hypersurface in hyperbolic space forms a singularity in finite time, is a sharp lower bound on the hyperbolic entropy of closed surfaces in $\mathbb{H}^3$.
\begin{cor}
	If $\Sigma\subset \mathbb{H}^3$ is a closed surface, then $\lambda_{\mathbb{H}}[\Sigma]> \lambda[\mathbb{S}^2]$.  Moreover, if $\Sigma$ has positive genus, then $\lambda_{\mathbb{H}}[\Sigma]> \lambda[\mathbb{S}^1\times \Real]$.
\end{cor}	
\begin{rem}
	One observes that $\lim_{r\downarrow 0} \lambda_{\mathbb{H}}[\partial B_{r}^{\mathbb{H}^3}(p)] =\lambda[\mathbb{S}^2]$ and so the lower bound is sharp even though it is not achieved.  
\end{rem}
In a similar manner, one may readily adapt the arguments from \cite{BernsteinWang1}, \cite{JZhu}, \cite{BWIsotopy} to see that for any closed hypersurface $\Sigma\subset \mathbb{H}^{n+1}$ one has $\lambda_{\mathbb{H}} [\Sigma]> \lambda[\mathbb{S}^n]$ and if $\Sigma\subset \mathbb{H}^4$ is closed and $\lambda_{\mathbb{H}}[\Sigma]\leq \lambda[\mathbb{S}^2\times \Real]$, then $\Sigma$ is isotopic to the boundary of a unit ball in $\mathbb{H}^4$.

We remark that in his thesis \cite{zhuGeometricVariationalProblems2018}, Zhu defines a notion of entropy, $\lambda_{\mathbb{S}}[\Sigma]$, of an $n$-dimensional submanifold, $\Sigma$, of $\mathbb{S}^{n+k}$ that is monotone along mean curvature flow in the sphere.  This is done by using the usual of embedding $\mathbb{S}^{n+k}$ into $\mathbb{R}^{n+k+1}$ to think of $\Sigma$ as a $n$-dimensional submanifold of $\Real^{n+k+1}$. He then defines the spherical entropy of $\Sigma$ to be its usual entropy, $\lambda[\Sigma]$, computed in the ambient Euclidean space.  We also refer the reader to \cite{sunEntropyClosedManifold2019} where an entropy is defined in arbitrary (closed) ambient manifolds that is monotone when the metric has parallel Ricci curvature and non-negative sectional curvature.  See also \cite{mramorEntropyGenericMean2018} where a notion of entropy in arbitrary Riemannian manifolds is obtained by isometrically embedding the ambient manifold in a large Euclidean space.

As a further application of Theorem \ref{MainThm}, we consider its implications for complete minimal submanifolds in hyperbolic space.  Specifically, we consider such minimal submanifolds that have a $C^1$-regular asymptotic boundary on the ideal boundary, $\partial_\infty \mathbb{H}^n$, of hyperbolic space. We remark that our computation seems to be related to the notion of \emph{holographic entanglement entropy} introduced by Ryu and Takayanagi in \cite{Ryu2006, Ryu2006a} -- see also \cite{Alexakis, Alexakis2010} for a mathematical treatment closer to our own. 

More precisely, recall that the Poincar\'{e} ball model of hyperbolic space $\mathbb{H}^n$ is the unit ball $\mathbb{B}^n=\set{\mathbf{x}: |\mathbf{x}|<1}\subset \Real^n$ with the Poincar\'{e} metric
$$
g_P=4\frac{ d\mathbf{x}\otimes d\mathbf{x}}{(1-|\mathbf{x}|^2)^2}.
$$
The ideal boundary,  $\partial_\infty \mathbb{H}^{n}$ of $\mathbb{H}^{n}$ is then identified with the sphere $\mathbb{S}^{n-1}=\partial \mathbb{B}^n$.   A $l$-dimensional submanifold $\Sigma\subset \mathbb{H}^{n}$ has \emph{$C^m$-regular asymptotic boundary}, provided that, after identifying with $\mathbb{B}^n$,  $\bar{\Sigma}\subset \bar{\mathbb{B}}^n$ is a $C^m$-regular submanifold with boundary and $\bar{\Sigma}$ meets $\partial\bar{\mathbb{B}}^n=\mathbb{S}^{n-1}$ orthogonally.  For such a $\Sigma$, set $\partial_\infty \Sigma=\partial \bar{\Sigma}=\bar{\Sigma}\cap  \mathbb{S}^{n-1}$.  Using the identification, we may think of $\partial_\infty \Sigma $ as a $C^m$-regular $(l-1)$-dimensional submanifold of $\partial_\infty \mathbb{H}^n$.  

As the isometries of $\mathbb{H}^n$ correspond to M\"{o}bius transforms of $\mathbb{B}^n$, this means $\partial_\infty \mathbb{H}^{n}$ possesses a natural conformal structure, but not a unique Riemannian structure.  In particular,  there is not a well defined notion of volume for submanifolds $\Gamma\subset \partial_\infty \mathbb{H}^n$.  However, there is a well defined notion of {conformal volume} in the spirit of Li and Yau \cite{Li1982}. 
Specifically, given a $l$-dimensional submanifold $\Gamma\subset \partial_\infty \mathbb{H}^{n}$, define the \emph{conformal volume} of $\Gamma$ to be
$$
\lambda_{c}[\Gamma]=\sup_{\psi\in \mathrm{Mob}(\mathbb{S}^{n-1})} Vol_{\mathbb{S}^{n-1}}(\psi(\Gamma)).
$$
Here $\mathrm{Mob}( \mathbb{S}^{n-1})$. is the group of M\"obius transformations of $\mathbb{S}^{n-1}$ and we have used the identification between $\mathbb{S}^{n-1}$ and $\partial_\infty \mathbb{H}^n$.  This is a well-defined geometric quantity and is essentially the $(n-1)$-conformal volume of the embedding of $\Gamma$ into $\mathbb{S}^{n-1}$ defined in \cite{Li1982}.

Given a complete $n$-dimensional submanifold with $C^1$-regular asymptotic boundary there is a natural relationship between its hyperbolic entropy and the conformal volume of its asymptotic boundary.
\begin{thm}\label{MobEntropyThm}
	If $\Sigma\subset \mathbb{H}^{n+k}$ is a complete $n$-dimensional submanifold with $C^1$-regular asymptotic boundary, $\partial_\infty \Sigma=\Gamma\subset \partial_\infty \mathbb{H}^{n+k}$, then
	$$
    	\lambda_{\mathbb{H}}[\Sigma]\geq\frac{\lambda_{c}[\Gamma]}{Vol_{\Real^n}(\mathbb{S}^{n-1})}.
	$$
	If, in addition,   $\Sigma$ is minimal and $n<N$ where $N$ is given by Theorem \ref{MainThm}, then this inequality is an equality.
\end{thm}
This is analogous to the relationship between the entropy of an asymptotically conical self-expander and the entropy of its asymptotic cone \cite[Lemma 3.5]{BWCompactness}. 
Together with the main result \cite{BWTopUniq} about topological properties of low entropy self-expanders, this suggests the  following conjecture:
\begin{conj}\label{MainConj}
	  If $\Gamma\subset \mathbb{S}^{2}$ is a closed curve with $\lambda_{c}[\Gamma]\leq \lambda[\mathbb{S}^{1}\times \Real]$ and $\Sigma_1,\Sigma_2\subset \mathbb{H}^{3}$ are  minimal surfaces with $C^1$-regular asymptotic boundaries that satisfy $\partial_\infty \Sigma_1=\partial_\infty \Sigma_2=\Gamma$, then $\Sigma_1$ is isotopic to $\Sigma_2$. 
\end{conj}
A consequence of this conjecture and \cite{andersonCompleteMinimalHypersurfaces1983} is that if $\Gamma\subset \mathbb{S}^2$ satisfies $\lambda_{c}[\Gamma]\leq \lambda[\mathbb{S}^{1}\times \Real]$ and $\Sigma\subset \mathbb{H}^3$ is a minimal surface with $\partial_\infty \Sigma=\Gamma$, then $\Sigma$ is a topological disk and, in particular, $\Gamma$ is connected.  The reason we restrict the conjecture to $n=2$, is that the natural analog of relative expander entropy that is used in the analysis of \cite{BWTopUniq} is the renormalized area (see \cite{grahamConformalAnomalySubmanifold1999, Alexakis2010}) and this quantity has special properties when $n=2$.    Nevertheless, it would be interesting see whether an analogous result held in higher dimensions -- for instance it may only hold in odd ambient dimensions.
	
\section{Properties of the heat kernel of $\mathbb{H}^n$}
We discuss here properties of the heat kernel on hyperbolic space that are relevant to the proof of Theorem \ref{MainThm} and Theorem \ref{MobEntropyThm}.  We make extensive use of the investigation of the heat kernel on hyperbolic space carried out by Davies and Mandouvalos in \cite{daviesHeatKernelBounds1988}.

Recall, the symmetries of $\mathbb{H}^n$ ensure that if $H_n(t,p; t_0, p_0)$ is the heat kernel on $\mathbb{H}^n$ with singularitity at $p=p_0$ and $t=t_0$, then, there is a function $K_n$ so that
$$
H_n(t,p; t_0, p_0)=K_n(t-t_0, \dist_{\mathbb{H}^n}(p,p_0)).
$$
For instance, as $\mathbb{H}^1$ is just the Euclidean line,
$$
K_1(t,\rho)=(4\pi t)^{-1/2}e^{-\frac{\rho^2}{4t}}
$$
and this gives the heat kernel in one-dimension.
Using this and \eqref{MillisonEqn}, one obtains that
$$
K_3(t,\rho)=(4\pi t)^{-3/2}\frac{\rho}{\sinh(\rho)}e^{-t-\frac{\rho^2}{4t}}
$$
gives the heat kernel on $\mathbb{H}^3$ -- see \cite{daviesHeatKernelBounds1988} for more details.

We recall two recurrence relations for $K_n$ from \cite{daviesHeatKernelBounds1988}.
The first is attributed to Millison:
\begin{equation}\label{MillisonEqn}
K_{n+2}(t,\rho) =-\frac{e^{-nt}}{2\pi \sinh(\rho)} \partial_\rho K_n(t,\rho).
\end{equation}
The second is:
\begin{equation} \label{IntegralEqn}
K_{n}(t,\rho)=\int_{\rho}^\infty \frac{e^{\frac{1}{4}\left(2n-1\right)t}K_{n+1}(t,s) \sinh(s)}{\left(\cosh(s)-\cosh(\rho)\right)^{\frac{1}{2}}} ds.
\end{equation}

Using these relations, Davies and Mandouvalos obtained the following uniform estimates on the $K_{n+1}$ \cite[Thereom 3.1]{daviesHeatKernelBounds1988}:
\begin{equation}\label{KnDecayEst}
K_{n+1}(t,\rho)\leq C_n t^{-\frac{1}{2}(n+1)}e^{-\frac{1}{4}n^2 t-\frac{\rho^2}{4t}-\frac{1}{2}n\rho} (1+\rho+t)^{\frac{1}{2}n -1} (1+\rho).
\end{equation}
In particular, for any fixed $p_0\in \mathbb{H}^{n+1}$ and $R>0$, on has for $t\geq 1$
\begin{equation}\label{TimeDecayEqn}
\sup_{p\in B_{R}^{\mathbb{H}^{n+1}}(p_0)} H_{n+1}(p,t; 0, p_0)\leq C_{n,R}' t^{-\frac{3}{2}} e^{-\frac{1}{4}n^2 t}.
\end{equation}

\begin{prop}\label{KnProp}
	If $\partial^2 \log K_n -\coth(\rho) \partial_\rho \log K_n \geq 0$ with strict inequality for $\rho>0$, then, for any $m\leq n$,
    \begin{equation}\label{KconvexityEqn}
	\partial^2 \log K_m -\coth(\rho) \partial_\rho \log K_m \geq 0 \mbox{ and the inequality is strict for $\rho>0$}.
	\end{equation} 
	In particular, \eqref{KconvexityEqn} holds for $1\leq m \leq 3$.
\end{prop}
\begin{proof}
	It is enough to show the result for $m=n-1$ as the first claim will then follow by induction.

  Clearly,
	$$
	\partial^2 \log K_m -\coth(\rho) \partial_\rho \log K_m \geq 0
	$$
	is equivalent to 
	$$
	\partial^2  K_m -\coth(\rho) \partial_\rho K_m \geq \frac{(\partial_\rho K_m)^2}{K_m}.
	$$
	Moreover, one inequality is strict only if the other is.
	We prove the second inequality.
	
	To that end, observe that combining \eqref{MillisonEqn} and \eqref{IntegralEqn} yields,
	$$
	\partial_\rho K_{n}(t,\rho)=\int_{\rho}^\infty \frac{e^{\frac{1}{4} (2n-1)t} (\partial_s K_{n+1})(t, s) \sinh(\rho)}{\left(\cosh (s)-\cosh(\rho)\right)^{\frac{1}{2}}} ds
	$$
	and
	\begin{align*}
	\partial_\rho^2 K_{n}(t, \rho)&= \int_{\rho}^\infty \frac{e^{\frac{1}{4} (2n-1)t} (\partial_s K_{n+1}(t, s) )\cosh(\rho)}{\left(\cosh (s)-\cosh(\rho)\right)^{\frac{1}{2}}} ds\\
	&+\int_{\rho}^\infty \frac{e^{\frac{1}{4} (2n-1)t} (\partial_s^2 K_{n+1}(t, s)-\coth(s) \partial_s K_{n+1}(t,s) )\frac{\sinh^2(\rho)}{\sinh(s)}}{\left(\cosh (s)-\cosh(\rho)\right)^{\frac{1}{2}}} ds
	\end{align*}
	As such,
	\begin{align*}
	\partial_\rho^2 K_{n-1}&-\coth(\rho) \partial_\rho K_{n-1} =\\
	&\int_{\rho}^\infty \frac{e^{\frac{1}{4} (2n-3)t} (\partial_s^2 K_{n}(t, s)-\coth(s) \partial_s K_{n}(t,s) )\frac{\sinh^2(\rho)}{\sinh(s)}}{\left(\cosh (s)-\cosh(\rho)\right)^{\frac{1}{2}}} ds\\
	&\geq \int_{\rho}^\infty \frac{e^{\frac{1}{4} (2n-3)t} \frac{(\partial_s K_{n})^2}{K_n}\frac{\sinh^2(\rho)}{\sinh(s)}}{\left(\cosh (s)-\cosh(\rho)\right)^{\frac{1}{2}}} ds\\
	\end{align*}
	Moreover, this inequality is strict for $\rho>0$ as can be seen by using the induction hypotheses and because $\sinh(\rho)>0$.  
	Hence, the Cauchy-Schwarz inequality gives
	\begin{align*}
	K_{n-1}&\left(\partial_\rho^2 K_{n-1}-\coth(\rho) \partial_\rho K_{n-1} \right) \\
	&\geq \int_{\rho}^\infty \frac{e^{\frac{1}{4} (2n-3)t} K_n \sinh(s)}{\left(\cosh (s)-\cosh(\rho)\right)^{\frac{1}{2}}} ds \int_{\rho}^\infty \frac{e^{\frac{1}{4} (2n-3)t} \frac{(\partial_s K_{n})^2}{K_n}\frac{\sinh^2(\rho)}{\sinh(s)}}{\left(\cosh (s)-\cosh(\rho)\right)^{\frac{1}{2}}} ds\\
	&\geq\left( \int_{\rho}^\infty \frac{e^{\frac{1}{4} (2n-3)t} (\partial_s K_{n}){\sinh(\rho)}}{\left(\cosh (s)-\cosh(\rho)\right)^{\frac{1}{2}}} ds\right)^2\\
	&=\left(\partial_\rho K_{n-1}\right)^2\\
	\end{align*}
	and the inequality is strict when $\rho>0$.
	That is,
	$$
	\partial_\rho^2 K_{n-1}-\coth(\rho) \partial_\rho K_{n-1} \geq \frac{(\partial_\rho K_{n-1})^2}{K_{n-1}}
	$$
	with strict inequality for $\rho>0$
	and the first claim is proved.  For the second, observe
	$$
	\log K_3(t,\rho)=- \frac{\rho^2}{4t}+\log(\rho)-\log(\sinh(\rho))-t-\frac{3}{2} \log(4\pi t)
	$$
    Hence,
    $$
    \partial_\rho \log K_3(t,\rho) =-\frac{\rho}{2t}+\frac{1}{\rho}-\coth(\rho)
    $$
    and
    $$
    \partial^2_\rho \log K_3(t,\rho)=- \frac{1}{2t} -\frac{1}{\rho^2} +\frac{1}{\sinh^2(\rho)}.
    $$
    It follows that
    $$
       \partial^2_\rho \log K_3-\coth(\rho)\partial_\rho \log K_3=\frac{ \rho \coth(\rho)-1}{2t} +\frac{1}{\sinh^2(\rho)}+\coth^2(\rho)-\frac{1}{\rho^2}-\frac{\coth(\rho)}{\rho}.
    $$   
    One readily checks that
    $$
    \rho \coth(\rho)-1\geq 0
    $$
    and the inequality is strict for $\rho>0$.
    Indeed, setting $f(\rho)=\rho \cosh(\rho)-\sinh(\rho)$ one has $f(0)=0$ and
    and
    $$
f'(\rho)=\rho \sinh(\rho)\geq 0
$$
with strict inequality for $\rho>0$.
Thus,  $f(\rho)> 0$ for $\rho> 0$. The inequality follows.
 Likewise, using power series:
    $$
    \frac{1}{\sinh^2(\rho)}+\coth^2(\rho)-\frac{1}{\rho^2}-\frac{\coth(\rho)}{\rho}\geq 0
    $$
    Indeed, 
    this is equivalent to
    $$
    g(\rho)= 1+\cosh^2(\rho)-\frac{\sinh^2(\rho)}{\rho^2}-\frac{\cosh(\rho)\sinh(\rho)}{\rho} \geq 0.
    $$
    Clearly, 
   \begin{align*}
    g(\rho)&=1 +\frac{1}{4}\left( e^{2\rho}+e^{-2\rho} +2\right) -\frac{1}{4\rho^2}\left( e^{2\rho}+e^{-2\rho} -2\right)-\frac{1}{4\rho} \left( e^{2\rho}-e^{-2\rho}\right)\\
    \end{align*}
    expanding as a power series gives
    \begin{align*}
    g(\rho)&= \frac{1}{4} \sum_{n=1}^\infty \left( \frac{(2^{n} +(-2)^n)}{n!}- \frac{(2^{n+2} +(-2)^{n+2})}{(n+2)!} -\frac{(2^{n+1}-(-2)^{n+1})}{(n+1)!}\right)\rho^{n}\\
    &=\frac{1}{4} \sum_{l=1}^\infty \left( 2\frac{4^{l}}{(2l)!} -8 \frac{4^{l}}{(2l+2)!}- 4\frac{4^{l}}{(2l+1)!}\right) \rho^{2l}\\
    &= \sum_{l=1}^\infty \left( 2 (2l+1)(2l+2) - 4(2l+2)-8\right) \frac{4^{l-1}}{(2l+2)!}\rho^{2l}\\     
    &=\sum_{l=1}^\infty \left(8l^2 + 4l-12\right) \frac{4^{l-1}}{(2l+2)!}\rho^{2l}.   
    \end{align*}
    One verifies that $8l^2+4l-12\geq 0$ for $l\geq 1$ and so $g(\rho) \geq 0$ as claimed.
    Hence,
    $$
     \partial^2_\rho \log K_3-\coth(\rho)\partial_\rho \log K_3\geq 0.
     $$
     with strict inequality when $\rho>0$.
     This completes the proof of the second claim.
\end{proof}


The non-strict convexity of Proposition \ref{KnProp} is equivalent to a certain convexity property for the heat kernel of $\mathbb{H}^n$.
\begin{prop} \label{HeatKernelConvexityEstProp}
	There is an $N\in [4,\infty]$ so if $n<N$, then the heat kernel $H_n(t,p; t_0, p_0)$ satisfies,  for every $t>t_0$,
	$$
	\nabla_{\mathbb{H}^n}^2 \log H_n - \coth(\rho) \partial_\rho \log H_n g_{\mathbb{H}^n} \geq 0.
	$$
	Here $\rho=\dist_{\mathbb{H}^n}(p,p_0)$.
\end{prop}
\begin{proof}

%
%
	
    The metric $g_{\mathbb{H}^n}$ has the form
	$$
	g_{\mathbb{H}^n}=d\rho \otimes d\rho +\sinh^2(\rho) g_{\mathbb{S}^{n-1}}.
	$$
	Hence, for any radial function $f=f(\rho(p))=f(\dist_{\mathbb{H}^n}(p,p_0))$ on $\mathbb{H}^n$
	$$
	\nabla^2_{\mathbb{H}^n} f=\partial^2_\rho f d\rho \otimes d\rho + \coth(\rho) \partial_\rho f  (g_{\mathbb{H}^n} -d\rho\otimes d\rho).
	$$
It follows that
	$$
	\nabla^2\log H_n-\coth(\rho) \partial_\rho H_n g_{\mathbb{H}^n}=(\partial^2_\rho \log K_n -\coth(\rho) \partial_\rho K_n) d\rho\otimes d\rho
	$$
	and so the convexity property is equivalent to
	$$
	\partial^2_\rho \log K_n -\coth(\rho) \partial_\rho K_n \geq 0.
	$$
	By Proposition \ref{KnProp}, once this condition holds for $N-1$ it holds for all $n<N$.  Moreover, it holds for $n=3$ and so we may take $N\geq 4$.
\end{proof}

\section{Proof of Theorem \ref{MainThm}}
In this section we prove the main montonicity result.
First of all, consider hyperbolic space $\mathbb{H}^{n+k}$ together with a distinguished point $p_0\in \mathbb{H}^{n+k}$.  Set $\rho=\dist_{\mathbb{H}^{n+k}}(p,p_0)$.   We use the following form of the metric in what follows
$$
g_{\mathbb{H}^{n+k}}=d\rho\otimes d\rho +\sinh^2(\rho) dh_{\mathbb{S}^{n+k-1}}.
$$
Using this form of the hyperbolic metric, it is straightforward to compute that, for any $C^2$ function $f=f(\rho(p))$ on $\mathbb{H}^{n+k}$ that is radial with respect to $p_0$,
\begin{equation}
\label{HessEqn}
\nabla^2_{\mathbb{H}^{n+k}} f=\partial^2_\rho f d\rho \otimes d\rho + \coth(\rho) \partial_\rho f  (g_{\mathbb{H}^{n+k}} -d\rho\otimes d\rho).
\end{equation}
Let $K_n(t,\rho)$ be the function from the previous section.  Set 
$$
\Phi_n^{t_0, p_0}(t,p)=K_n(t_0-t, \rho(p)).
$$
When $k=0$,
$$
\Phi_n^{t_0, p_0}(t,p)=K_n(t_0-t,p)
$$
so,  in this case,  $\Phi_n^{t_0, p_0}$ is the backwards heat kernel on $\mathbb{H}^n$.
For general $k$, taking the trace of \eqref{HessEqn} implies,
$$
\left(\frac{\partial}{\partial t} +\Delta_{\mathbb{H}^{n+k}}\right) \Phi_n^{t_0, p_0}=k \coth(\rho) \partial_\rho \Phi_n^{t_0, p_0}.
$$
We have the following analog of Huisken's formula:
\begin{prop}\label{MonotonicityFormulaProp}
Suppose $\set{\Sigma_t}_{t\in [0,T)}$ is a mean curvature flow in $\mathbb{H}^{n+k}$ of $n$-dimensional submanifolds with exponential volume growth. For any $t_0\in (0,T]$, $p_0\in \mathbb{H}^{n+k}$ and $t\in (0,t_0)$ one has
\begin{align*}
\frac{d}{dt} \int_{\Sigma_t} \Phi_n^{t_0,p_0} dVol_{\Sigma_t} &=- \int_{\Sigma_t}\left( \left| \frac{\nabla^\perp_{\Sigma_t} \Phi_n^{t_0,p_0}}{\Phi_n^{t_0,p_0}} -\mathbf{H}_{\Sigma_t}\right|^2 +Q(\Phi_n^{t_0,p_0}) \right) \Phi_n^{t_0,p_0} dVol_{\Sigma_t}.
\end{align*}
Here,
$$
Q(\Phi_n^{t_0,p_0})=\sum_{i=1}^{k}\nabla^2_{\mathbb{H}^{n+k}}\log \Phi_n^{t_0, p_0} (\nu_i, \nu_i) - k \coth(\rho) \partial_\rho \log \Phi_n^{t_0, p_0}
$$
and $\set{\nu_1(p), \ldots, \nu_{k}(p)}$ is any choice of orthonormal basis of $N_p\Sigma_t$, the normal space of $\Sigma_t$ at $p$.  
\end{prop}
\begin{proof}
	By the first variation formula we have
	\begin{align*}
	\frac{d}{dt} \int_{\Sigma_t}& \Phi_n^{t_0, p_0} dVol_{\Sigma_t}=\int_{\Sigma_t} \frac{\partial}{\partial_t} \Phi_n^{t_0, p_0}+\nabla_{\mathbb{H}^{n+k}} \Phi_n^{t_0, p_0} \cdot \mathbf{H}_{\Sigma_t} -|\mathbf{H}_{\Sigma_t}|^2 \Phi_n^{t_0, p_0}  dVol_{\Sigma_t}
	\end{align*}
	As  $\Sigma_t$ has exponential volume growth,  \eqref{KnDecayEst} and \eqref{MillisonEqn} can be used together with an integration by parts to see
	$$
	\int_{\Sigma_t} \Delta_{{\Sigma_t}} \Phi_n^{t_0, p_0} dVol_{\Sigma_t}=0.
	$$
	Standard geometric computations imply that on $\Sigma_t$,
	$$
	\Delta_{\mathbb{H}^{n+k}} \Phi_n^{t_0, p_0} =  \Delta_{{\Sigma_t}}\Phi_n^{t_0, p_0} -\mathbf{H}_{\Sigma_t}\cdot \nabla_{\mathbb{H}^{n+k}}  \Phi_n^{t_0, p_0}+\sum_{i=1}^{k}\nabla^2_{\mathbb{H}^{n+k} }\Phi_n^{t_0, p_0} (\nu_i, \nu_i)
	$$
	where $\set{\nu_1(p), \ldots, \nu_{k}(p)}$ is an orthonormal basis of $N_p\Sigma_t$.  
	Hence,
	\begin{align*}
	\left(\frac{\partial}{\partial t} +\Delta_{\Sigma_t}\right)\Phi_n^{t_0, p_0} & =	\left(\frac{\partial}{\partial t} +\Delta_{\mathbb{H}^{n+k}}\right)\Phi_n^{t_0, p_0}+\mathbf{H}_{\Sigma_t}\cdot \nabla_{\mathbb{H}^{n+k}}  \Phi_n^{t_0, p_0}\\
	& -\sum_{i=1}^{k}\nabla^2_{\mathbb{H}^{n+k} }\Phi_n^{t_0, p_0} (\nu_i, \nu_i)\\
	\end{align*}
	It follows that
\begin{align*}
	\left(\frac{\partial}{\partial t} +\Delta_{\Sigma_t}\right)\Phi_n^{t_0, p_0} 	&= 	k \coth(\rho) \partial_\rho \Phi_{n}^{t_0,p_0}+\mathbf{H}_{\Sigma_t}\cdot \nabla_{\mathbb{H}^{n+k}}  \Phi_n^{t_0, p_0}\\
	& -\sum_{i=1}^{k}\nabla^2_{\mathbb{H}^{n+k} }\Phi_n^{t_0, p_0} (\nu_i, \nu_i)\\
	&=\mathbf{H}_{\Sigma_t}\cdot \nabla_{\mathbb{H}^{n+k}}  \Phi_n^{t_0, p_0} -Q(\Phi_n^{t_0,p_0}) \Phi_n^{t_0, p_0}	-\frac{|\nabla_{\Sigma_t}^\perp \Phi_n^{t_0,p_0}|^2}{\Phi_n^{t_0, p_0}}
	\end{align*}
	Where we used that
	$$
	 \nabla_{\Sigma_t}^2 \log \Phi_n^{t_0,p_0}=\frac{\nabla_{\Sigma_t}^2\Phi_n^{t_0,p_0}}{\Phi_n^{t_0,p_0}} -\frac{d\Phi_n^{t_0,p_0} \otimes d\Phi_n^{t_0,p_0}}{(\Phi_n^{t_0,p_0})^2}.
	$$
	Hence,
	\begin{align*}
	 \frac{d}{dt} \int_{\Sigma_t}& \Phi_n^{t_0, p_0} dVol_{\Sigma_t}=\int_{\Sigma_t}  \left(\frac{\partial}{\partial t} +\Delta_{\Sigma_t} \right)\Phi_n^{t_0, p_0}+\nabla_{\mathbb{H}^{n+k}} \Phi_n^{t_0, p_0} \cdot \mathbf{H}_{\Sigma_t} \\
	 &-|\mathbf{H}_{\Sigma_t}|^2\Phi_n^{t_0, p_0}  dVol_{\Sigma_t}\\
	 &=- \int_{\Sigma_t}\left( \left|\frac{\nabla_{\Sigma_t}^\perp \Phi_n^{t_0,p_0}}{\Phi_n^{t_0,p_0}} -\mathbf{H}_{\Sigma_t} \right|^2 +Q(\Phi_n^{t_0,p_0}) \right) \Phi_n^{t_0,p_0}dVol_{\Sigma_t}.
	\end{align*}
	This completes the proof.
\end{proof}

It is helpful to have the following elementary geometric lemma.
\begin{lem}\label{ConeLem}
Let $\Sigma \subset \mathbb{H}^{n+k}$ be a complete $n$-dimensional submanifold and let $p_0\in \Sigma$ and $\rho=\dist_{\mathbb{H}^{n+k}}(p,p_0)$. If $|\nabla_\Sigma^\perp \rho|=0$ for $p\in \Sigma\setminus \set{p_0}$, then $\Sigma$ is a cone over $p_0$.
\end{lem}
\begin{proof}
By \eqref{HessEqn},
	$$
	\nabla_{\mathbb{H}^{n+k}}^2 \cosh(\rho)=\cosh(\rho) g_{\mathbb{H}^{n+k}}
	$$
 and so the condition that $\nabla_\Sigma^\perp \rho=0$ ensures $\nabla_\Sigma^\perp \cosh (\rho)=0$ and so
	$$
	\nabla_{\Sigma}^2 \cosh(\rho)= \cosh(\rho) g_{\Sigma}.
	$$
	Moreover, $\Sigma$ is connected.  Indeed, if $\Sigma$ were not connected, then there would be a component $\Sigma'$ that does not contain $p_0$.  As $\Sigma$ is complete, this implies the existence of a point $p'\in \Sigma'$ at which $\rho$ achieves its (non-zero) minimum.  As this is a regular point of $\rho$, one must have $\nabla_{\Sigma}\rho=0$ at this point and so $|\nabla^\perp_\Sigma \rho|(p')=1$, a contradiction.

Hence, for each $p\in \Sigma\setminus p_0$, there is a minimizing geodesic parameterized by arclength $\gamma:[0,L]\to \Sigma$ connecting $p_0$ to $p$. That is, $\gamma$ satisfies $\gamma(0)=p_0$ and $\gamma'(s)$ of unit length. 
	One computes, 
	$$
	\frac{d^2}{ds^2} \cosh(\rho(\gamma(s)))=\nabla^2_{\Sigma} \cosh(\rho) (\gamma'(s), \gamma'(s))=\cosh\rho(\gamma(s)).
	$$
	As $\cosh(\rho(\gamma(s)))=1$ and $\frac{d}{ds}|_{s=0} \cosh(\rho(\gamma(s)))=0$, one has, by standard ODE analysis, that
	$$
	\cosh(\rho(\gamma(s)))=\cosh(s)
	$$
	hence, $\rho(\gamma(s))=s$ and so $\nabla_{\Sigma} \rho(\gamma(s))=\gamma'(s)$.  This implies $L=\rho(p)=\dist_{\mathbb{H}^{n+k}}(p,p_0)$ so $\gamma$ is a minimizing geodesic in $\mathbb{H}^{n+k}$.  As $p$ was arbitrary this implies $\Sigma$ is a cone.
\end{proof}

 We are now ready to prove Theorem \ref{MainThm}.
 \begin{proof}[Proof of Theorem \ref{MainThm}]
   As $\Phi_n^{t_0,p_0}$ is radial with respect to $p_0$, \eqref{HessEqn} implies that at $p\in \Sigma_t$,
\begin{align*}
 Q(\Phi_{n}^{t_0,p_0}) &=  \sum_{i=1}^k \left(  \partial^2_\rho \log K_n (d\rho(\nu_i))^2  +\coth(\rho) \partial_\rho \log K_n \left( 1-(d\rho(\nu_i))^2\right) \right)\\
 &-k\coth(\rho) \partial_\rho \log K_n\\
 &=
 \left(\partial^{2}_\rho \log K_n-\coth(\rho) \partial_\rho \log K_n\right) |\nabla_{\Sigma_t}^\perp \rho|^2
\end{align*}
where here $\set{\nu_1(p), \ldots, \nu_k(p)}$ is an orthonormal basis of $N_p \Sigma_t$.
  Thus, by Proposition \ref{KnProp} there is an $N\in [4,\infty]$ so for $n<N$, 
  $$
 Q(\Phi_{n}^{t_0,p_0})(t,p)\geq 0
  $$
  and the inequality is strict unless $p=p_0$ or $|\nabla_{\Sigma_t}^\perp \rho|(p)=0$.  
  Hence,
  $$
  \int_{\Sigma_t} Q(\Phi_{n}^{t_0,p_0}) \Phi_n^{t_0,p_0}\geq 0
  $$
  with strict inequality unless $\nabla_{\Sigma_t}^\perp \rho$ is identically zero.  By Lemma \ref{ConeLem}, this occurs only if $\Sigma_t$ is a cone over $p_0$.
  
  Hence, by Proposition \ref{MonotonicityFormulaProp},
  $$ 
  \frac{d}{dt} \int_{\Sigma_t} \Phi_n^{t_0,p_0}(t,p) dVol_{\Sigma_t}\leq 0
  $$
  and the inequality is strict unless $\Sigma_t$ is a cone over $p_0$.  If $\Sigma_t$ is a cone over $p_0$, then $\nabla^\perp_{\Sigma_t} \Phi_n^{t_0,p_0}=0$ and so, by Proposition \ref{MonotonicityFormulaProp},
  $$ 
  \frac{d}{dt} \int_{\Sigma_t} \Phi_n^{t_0,p_0}(t,p) dVol_{\Sigma_t}=-\int_{\Sigma_t} |\mathbf{H}_{\Sigma_t}|^2 \Phi_n^{t_0,p_0} dVol_{\Sigma_t}.
  $$
  In particular, the inequality is strict unless $\Sigma_t$ is a minimal submanifold that is a cone over $p_0$.  This completes the proof.
 \end{proof}
 
 We remark that while we carried out these computations for smooth mean curvature flows, they carry over to weak flows in the sense of Brakke as in \cite{IlmanenMon}.

\section{Hyperbolic Entropy of complete minimal submanifolds}
In this section we estimate from below the hyperbolic entropy of complete submanifolds with $C^1$-regular asymptotic boundary on the ideal boundary, $\partial_\infty \mathbb{H}^n$, in terms of a certain boundary entropy -- i.e., the conformal volume. These computations yield an exact formula when the submanifold is minimal of a dimension for which Theorem \ref{MainThm} holds.

Let us expand on the notation from the introduction.  First, recall the Poincar\'{e} ball model of hyperbolic space, $\mathbb{H}^n$ is the open unit ball in Euclidean space 
$$
\mathbb{B}^n=\set{\mathbf{x}:|\mathbf{x}|<1}\subset \Real^{n}
$$
together with the Poincar\'{e} metric
$$
g_P=4\frac{d\mathbf{x}\otimes d\mathbf{x}}{(1-|\mathbf{x}|^2)^2}=\frac{4}{(1-|\mathbf{x}|^2)^2} g_{\Real^n}.
$$
That is, for any model of hyperbolic space, $(\mathbb{H}^n,g_{\mathbb{H}^n})$, there is an isometry $i: \mathbb{H}^n\to \mathbb{B}^n$ so $i^* g_{P}=g_{\mathbb{H}^n}$. 
The isometries of $g_P$ are given by the M\"{o}bius transforms of $\mathbb{B}^n$ and so this identification is not unique.  In fact, for any point $p_0\in \mathbb{H}^n$, there is an isometry $i:  \mathbb{H}^n\to \mathbb{B}^n$ so $i(p_0)=0$.  Moreover, if $i, j:  \mathbb{H}^n\to \mathbb{B}^n$ satisfy $i(p_0)=j(p_0)=0$, then $i\circ j^{-1}$ is an orthogonal transformation of $\mathbb{B}^n$.  In particular, in this case $i^*g_{\Real^n}=j^*g_{\Real^n}$ while these metrics are different for identifications associated to distinct distinguished points.   In what follows, we will always choose a distinguished point $p_0\in \mathbb{H}^n$ and an identification (i.e., an isometry) $i: \mathbb{H}^n\to \mathbb{B}^n$ with $i(p_0)=0$. 
We use this identification to compactify $\mathbb{H}^n$ and denote the \emph{ideal boundary} of $\mathbb{H}^n$ by $\partial_\infty \mathbb{H}^n$ which is identified with $\mathbb{S}^{n-1}=\partial \mathbb{B}^n$ in the natural way. Extend $i:\bar{\mathbb{H}}^n\to \bar{\mathbb{B}}^n$ in the obvious way.  This compactification is independent, as a manifold with boundary, of the choice of $p_0$ and $i$.

A complete submanifold $\Sigma\subset \mathbb{H}^n$ has \emph{$C^m$-regular asymptotic boundary} for $1\leq m\leq\infty$ if $\Sigma'=i(\Sigma) \subset \mathbb{B}^n$ has the property that its closure $\bar{\Sigma}'$ is a $C^m$-regular manifold with boundary, $\partial \bar{\Sigma}'\subset \mathbb{S}^{n-1}=\partial \mathbb{B}^n$ and $\bar{\Sigma}'$ meets $\mathbb{S}^{n-1}$ orthogonally. Denote by $\partial_\infty \Sigma$ the submanifold corresponding to $\partial \bar{\Sigma}'$ in $\partial_\infty \mathbb{H}^n$.   As M\"{o}bius transformations are smooth and conformal this is a well defined notion independent of choice of identification.

Using the identification, $i$, $\partial_{\infty} \mathbb{H}^n$ has a well defined Riemannian metric induced from $\mathbb{S}^{n-1}=\partial \mathbb{B}^n$.  While this metric depends on $p_0$, it is otherwise independent of the choice of isometry taking $p_0$ to $0$. Let us denote this metric by $g_{\partial_\infty \mathbb{H}^n}^{p_0}$. 
  Clearly, $g_{\partial_\infty \mathbb{H}^{n}}^{p_0}$ and $g_{\partial_\infty \mathbb{H}^{n}}^{q_0}$ are conformal for different choices of distinguished point $p_0$ and $q_0$ and so $\partial_\infty \mathbb{H}^n$ has a well defined conformal structure. In fact, the two metrics are related by a M\"{o}bius transform on the sphere.   Fix a $l$-dimensional $C^m$ submanifold of $\partial_\infty \mathbb{H}^n$ and let $i(\Gamma)\subset \mathbb{S}^{n-1}$ be the corresponding submanifold of the sphere under the identification.  Set
$$
Vol_{\partial_\infty \mathbb{H}^n}(\Gamma,p_0)=Vol_{\mathbb{S}^{n-1}}(i(\Gamma))
$$
If $q_0$ is a different choice of distinguished point, then, in general, it can happen that
$$
Vol_{\partial_\infty \mathbb{H}^n}(\Gamma,p_0)\neq Vol_{\partial_\infty \mathbb{H}^n}(\Gamma,q_0).
$$ 
However, there is a M\"{o}bius transform, $\psi\in \mathrm{Mob}(\mathbb{S}^{n-1})$ so that
$$
Vol_{\partial_\infty \mathbb{H}^n}(\Gamma,q_0)=Vol_{\mathbb{S}^{n-1}}(\psi(i(\Gamma)))
$$
Hence, if we define the \emph{conformal volume}  of $\Gamma\subset \partial_\infty \mathbb{H}^n$ by
$$
\lambda_{c}[\Gamma]=\sup_{\psi\in \mathrm{Mob}(\mathbb{S}^{n-1})} Vol_{\mathbb{S}^{n-1}}(\psi(i(\Gamma))),
$$
then this is a quantity that is well defined independent of the choice distinguished point and of identification.  Indeed,
$$
\lambda_c[\Gamma]=\sup_{p_0\in \mathbb{H}^n} Vol_{\partial_\infty \mathbb{H}^n}(\Gamma,p_0).
$$

One has the following useful calculation that clarifies the meaning of $Vol_{\partial _\infty \mathbb{H}^{n+k}}(\Gamma)$ when $\Gamma$ appears as the asymptotic boundary of a $n$-dimensional submanifold.
\begin{lem}\label{BoundaryLem}
Let $\Sigma \subset \mathbb{H}^{n+k}$ be a $n$-dimensional submanfiold that has $C^1$-regular asymptotic boundary. 
For any $p_0\in \mathbb{H}^{n+k}$,
$$
Vol_{\partial_\infty \mathbb{H}^{n+k}}(\partial_\infty \Sigma; p_0)=\lim_{r\to \infty} \frac{Vol_{\mathbb{H}^{n+k}}(\Sigma\cap \partial B_{r}^{\mathbb{H}^{n+k}}(p_0))}{\sinh^{n-1}(r)}
$$
\end{lem}
\begin{proof}
 Pick an identification $i:\mathbb{H}^{n+k}\to \mathbb{B}^{n+k}$ so that $i(p_0)=0$.  As $i^{*}g_P=g_{\mathbb{H}^{n+k}}$, $i$ is an isometry and so
 $i(\partial B_r^{\mathbb{H}^{n+k}})=\partial B_{r}^{g_P}(0)$.  Furthermore, as the conformal factor of $g_P$ is radial, $\partial B_{r}^{g_P}(0)=\partial B_{s}(0)$ where $r=\ln \left( \frac{1+s}{1-s}\right) $.
 
 Set $\Sigma_r=\Sigma\cap \partial B_{r}^{\mathbb{H}^{n+k}}(p_0)$. 
 and let $\Sigma'=i(\Sigma)$. From the above,   $i(\Sigma_r)=\Sigma'\cap \partial B_s(0)=\Sigma'_s$.  Let $g_r$ be the metric on $\Sigma_r$ induced from $\mathbb{H}^{n+k}$ and $g_{s}'$ be the metric induced on $\Sigma'_s$ from $g_{\Real^{n+k}}$.  Clearly,
 $$
 i^*g_{s}'= \sinh^{-2}(r)(g_r)
 $$
 In particular, as $\Sigma_s'$ is $n-1$-dimensional,
 $$
 Vol_{\Real^{n+k}}(\Sigma_s')= \frac{Vol_{\mathbb{H}^{n+k}}(\Sigma_r)}{\sinh^{n-1}(r) }.
 $$
 The definition of $C^1$-regular asymptotic boundary ensures
 $$
 \lim_{s\to 1} Vol_{\Real^{n+k}}(\Sigma_s')=Vol_{\mathbb{S}^{n+k-1}}(\partial \Sigma' )= Vol_{\partial_\infty \mathbb{H}^{n+k}}(\partial_\infty \Sigma; p_0).
 $$
As $s\to 1$,  $r\to \infty$ and so
 \begin{align*}
 Vol_{\partial_\infty \mathbb{H}^{n+k}}(\partial_\infty \Sigma; p_0) &=
 \lim_{r\to \infty}   \frac{Vol_{\mathbb{H}^{n+k}}(\Sigma_r)}{\sinh^{n-1}(r)}.
 \end{align*}
 This completes the proof.
\end{proof}

Theorem \ref{MobEntropyThm} is a consequence of Theorem \ref{MainThm} and the following proposition:
\begin{prop}\label{LimitProp}
Let $\Sigma\subset \mathbb{H}^{n+k}$ be an $n$-dimensional submanifold with $C^1$-regular asymptotic boundary. 
One has
$$
\lim_{t\to -\infty} \int_{\Sigma} \Phi^{t_0, p_0} (t,p) dVol_{\Sigma}(p)=\frac{Vol_{\partial_\infty \mathbb{H}^{n+k}}(\partial_\infty \Gamma; p_0)}{Vol_{\Real^{n}}(\mathbb{S}^{n-1})}.
$$
\end{prop}
\begin{proof}
 Let $\Sigma_r=\Sigma\cap \partial B_{r}^{\mathbb{H}^{n+k}}(p_0)$.  Observe that, by the definition of having a $C^1$-regular asymptotic boundary there is an $R_0>0$ so that, for $r\geq R_0$, $\Sigma $ meets $\partial B_r^{\mathbb{H}^{n+k}}(p_0)$ transversally and so $\Sigma_r$ is a smooth $(n-1)$-dimensional submanifold of  $\partial B_r^{\mathbb{H}^{n+k}}(p_0)$.  
 
It follows from Lemma \ref{BoundaryLem} that, for any $\epsilon>0$, there is an $R_\epsilon>R_0$ so for $r>R_\epsilon$
 $$
 (1-\epsilon) { Vol_{\partial_\infty \mathbb{H}^{n+k}}(\partial_\infty \Sigma; p_0)}\leq  \frac{ Vol_{\mathbb{H}^{n+k}}(\Sigma_r)}{{\sinh^{n-1}(r)}}\leq (1+\epsilon) {Vol_{\partial_\infty \mathbb{H}^{n+k}}(\partial_\infty \Sigma; p_0)}.
 $$
 We also observe, that if $\rho=\dist_{\mathbb{H}^{n+k}}(p,p_0)$, then $|\nabla_{\mathbb{H}^{n+k}} \rho|=1$ for $p\neq p_0$.  It follows from the definition of having a $C^1$-regular asymptotic boundary -- specifically the fact that $\Sigma$ is $C^1$ up to $\partial_\infty \mathbb{H}^{n+k}$ and meets this boundary orthogonally, that
 $$
 \lim_{p\to \infty} |\nabla_\Sigma \rho|=1.
 $$
 In particular, there is a $R_\epsilon'>R_\epsilon$ so that for $p\in \Sigma\setminus B_{R_\epsilon'}^{\mathbb{H}^{n+k}}(p_0)$,
 $$
 1\leq \frac{1}{|\nabla_\Sigma \rho|}\leq 1+\epsilon.
 $$
 
 Hence, by the co-area formula, for $R>R_\epsilon'$, one has
 \begin{align*}
 (1-\epsilon) Vol_{\partial_\infty\mathbb{H}^{n+k}}(\partial_\infty \Sigma; p_0) &
 \int_R^\infty K_n(t_0-t, r)  \sinh^{n-1}(r) dr\\
 & \leq \int_{\Sigma\setminus B_R^{\mathbb{H}^{n+k}}(p_0)} \Phi_n^{t_0,p_0}(t,p) dVol_\Sigma(p)
\end{align*}
 and
\begin{align*}
 \int_{\Sigma\setminus B_R^{\mathbb{H}^{n+k}}(p_0)} & \Phi_n^{t_0,p_0}(t,p) dVol_\Sigma(p)\\
 &\leq   (1+\epsilon)^2 Vol_{\partial_\infty\mathbb{H}^{n+k}}(\partial_\infty \Sigma; p_0) \int_R^\infty K_n(t_0-t, r)  \sinh^{n-1}(r) dr.
\end{align*}
 
 Now suppose for the moment that $\Sigma=\mathbb{H}^n$ where $\mathbb{H}^n$ is the totally geodesic hyperbolic plane going through $p_0$. As $ \Phi_n^{t_0,p_0}(t,p) $ restricts to the backwards heat kernel on $\mathbb{H}^n$, one has, for all $t<t_0$,
 $$
 \int_{\mathbb{H}^n} \Phi_n^{t_0,p_0}(t,p) dVol_{\mathbb{H}^n}(p)=1.
 $$
 By \eqref{TimeDecayEqn}, the uniform decay in time of the heat kernel on bounded sets,  one has
 $$
 \lim_{t\to -\infty} \int_{\mathbb{H}^n\cap \bar{B}_R^{\mathbb{H}^{n+k}}(p_0)}  \Phi_n^{t_0,p_0}(t,p) dVol_{\mathbb{H}^n}(p)=0
 $$
 Hence, for any $R>0$,
 $$
 \lim_{t\to -\infty} \int_{\mathbb{H}^n\setminus {B}_R^{\mathbb{H}^{n+k}}(p_0)}  \Phi_n^{t_0,p_0}(t,p) dVol_{\mathbb{H}^n}(p)=1.
 $$
The co-area formula and fact that $\mathbb{H}^n$ is a cone over $p_0$ implies
 \begin{align*}
  \int_{\mathbb{H}^n\setminus {B}_R^{\mathbb{H}^{n+k}}(p_0)} & \Phi_n^{t_0,p_0}(t,p) dVol_{\mathbb{H}^n}(p) =\int_R^\infty K_n(t_0-t, r) Vol_{\mathbb{H}^{n}}(\partial {B}_r^{\mathbb{H}^{n}}(p_0)) dr\\
&=  Vol(\partial_\infty \mathbb{H}^n; p_0)\int_R^\infty K_n(t_0-t, r) \sinh^{n-1}(r) dr\\
  &=Vol_{\Real^{n}}(\mathbb{S}^{n-1})\int_R^\infty K_n(t_0-t, r) \sinh^{n-1}(r) dr.
\end{align*}
Here the final equality used that $\mathbb{H}^n$ goes through $p_0$.
Hence, for any $R>0$, 
$$
\lim_{t\to -\infty} \int_R^\infty K_n(t_0-t, r) \sinh^{n-1}(r) dr=\frac{1}{Vol_{\Real^{n}}(\mathbb{S}^{n-1})}.
$$
It follows that,
\begin{align*}
 (1-\epsilon) \frac{ Vol(\partial_\infty \Sigma; p_0) }{Vol_{\Real^{n}}(\mathbb{S}^{n-1})}&\leq \liminf_{t\to -\infty} \int_{\Sigma\setminus B_R^{\mathbb{H}^{n+k}}(p_0)} \Phi_n^{t_0,p_0}(t,p) dVol_\Sigma(p)\\
 &=\liminf_{t\to -\infty} \int_{\Sigma} \Phi_n^{t_0,p_0}(t,p) dVol_\Sigma(p)\\
\end{align*}
 where the second equality again uses \eqref{TimeDecayEqn}.
Hence, we may send $\epsilon \to 0$ and obtain,
$$
 \frac{ Vol(\partial_\infty \Sigma; p_0) }{Vol_{\Real^{n}}(\mathbb{S}^{n-1})}\leq \liminf_{t\to -\infty} \int_{\Sigma} \Phi_n^{t_0,p_0}(t,p) dVol_\Sigma(p)
$$
In the exact same way, one obtains
$$
 \limsup_{t\to -\infty} \int_{\Sigma} \Phi_n^{t_0,p_0}(t,p) dVol_\Sigma(p)\leq
\frac{ Vol(\partial_\infty \Sigma; p_0) }{Vol_{\Real^{n}}(\mathbb{S}^{n-1})}
$$
This proves that
$$
\lim_{t\to -\infty} \int_{\Sigma} \Phi_n^{t_0,p_0}(t,p) dVol_\Sigma(p)=
\frac{ Vol(\partial_\infty \Sigma; p_0) }{Vol_{\Real^{n}}(\mathbb{S}^{n-1})}
$$
verifying the main claim.  
\end{proof}

We may now prove Theorem \ref{MobEntropyThm}.
\begin{proof}[Proof of Theorem \ref{MobEntropyThm}.]
By definition, for any fixed $p_0\in \mathbb{H}^{n+k}$,
$$
\lambda_{\mathbb{H}}[\Sigma]\geq \limsup_{t\to -\infty} \int_{\Sigma} \Phi^{0, p_0}_n (t,p) dVol_{\Sigma}(p).
$$
Hence, Proposition \ref{LimitProp} implies
$$
\lambda_{\mathbb{H}}[\Sigma]\geq  \frac{Vol_{\partial_\infty \mathbb{H}^n}(\partial_\infty \Gamma; p_0)}{Vol_{\Real^{n}}(\mathbb{S}^{n-1})}.
$$
Taking the supremum over $p_0\in\mathbb{H}^{n+k}$ yields, 
$$
\lambda_{\mathbb{H}}[\Sigma]\geq\frac{\lambda_{c}[\Gamma]}{Vol_{\Real^{n}}(\mathbb{S}^{n-1})}.
$$
This proves the first claim.  To see the second claim, observe that if $\Sigma$ is minimal, then we may think of $\Sigma$ as a static solution of mean curvature flow. It readily follows from Lemma \ref{BoundaryLem} that $\Sigma$ has exponential volume growth. Hence, by Theorem \ref{MainThm}, for all $\tau>0$
\begin{align*}
 \int_{\Sigma} \Phi^{0, p_0} (-\tau,p) dVol_{\Sigma}(p) &\leq \lim_{t\to -\infty} \int_{\Sigma} \Phi^{0, p_0} (t-\tau,p) dVol_{\Sigma}(p)\\
 &= \frac{Vol_{\partial_\infty \mathbb{H}^{n+k}}(\partial_\infty \Gamma; p_0)}{Vol_{\Real^{n}}(\mathbb{S}^{n-1})}\leq \frac{\lambda_{c}[\Gamma]}{Vol_{\Real^{n}}(\mathbb{S}^{n-1})}.
\end{align*}
Taking the supremum over $\tau>0$ and $p_0\in \mathbb{H}^{n+k}$ yields
$$
\lambda_{\mathbb{H}}[\Sigma]\leq \frac{\lambda_{c}[\Gamma]}{Vol_{\Real^{n}}(\mathbb{S}^{n-1})}.
$$
Combined with the first claim this completes the proof.
\end{proof}

Finally, it is enlightening to introduce the function 
$$
s=2\frac{1-|\mathbf{x}|}{1+|\mathbf{x}|}
$$
on $\mathbb{B}^{n+k} \backslash\set{0}$. One verifies that $s$ is a boundary defining function and,
in appropriate associated  coordinates,  the Poincar\'{e} metric has the form
$$
g_P= s^{-2} \left( ds^2+ \left(1-\frac{s^2}{4}\right)^{2} g_{\mathbb{S}^{n+k-1}}\right).
$$
Hence, if $i$ is an identification of $\mathbb{H}^{n+k}$ with $\mathbb{B}^{n+k}$ sending $p_0$ to $0$ and $\sigma=s\circ i$, then $\sigma^2 g_{\mathbb{H}^{n+k}}$ is a conformal compactification in the sense of \cite{grahamConformalAnomalySubmanifold1999}.
Suppose $\Sigma\subset \mathbb{H}^{n+k}$ is a $n$-dimensional minimal submanifold that has $C^\infty$-regular asymptotic boundary. Using a boundary defining function like $\sigma$,  Graham and Witten \cite{grahamConformalAnomalySubmanifold1999} showed that, when $n$ is even,
\begin{align*}
Vol_{\mathbb{H}^{n+k}}(\Sigma\cap B_{r(\epsilon)}(p_0))&=Vol_{\mathbb{H}^{n+k}}(\Sigma\cap \set{\sigma>\epsilon})\\
&=c_0 \epsilon^{-n+1}+c_2 \epsilon^{-n+3}+\cdots+c_1 \epsilon^{-1} +c_0 +o(1)
\end{align*}
while if $n$ is odd, then
\begin{align*}
Vol_{\mathbb{H}^{n+k}}(\Sigma\cap B_{r(\epsilon)}(p_0))&=
Vol_{\mathbb{H}^{n+k}}(\Sigma\cap \set{\sigma>\epsilon})\\
&= c_0 \epsilon^{-n+1}+c_2 \epsilon^{-n+3}+\cdots+c_1 \epsilon^{-1} +d \log \frac{1}{\epsilon} +c_0 +o(1).
\end{align*}
Here $r(\epsilon)=-\ln\left(\frac{\epsilon}{2}\right)$.
They further showed that, when $n$ is even, $c_0$ is independent of the choice of $\sigma$ (and so independent of choice of $p_0$)  and the same is true of $d$ when $n$ is odd. When $n=2$, this quantity is precisely the renormalized area considered in \cite{Alexakis2010}. This expansion is also the (Riemannian) analog of the entropy considered by Ryu and Takayanagi \cite{Ryu2006, Ryu2006a}. One observes that
$$
c_0=Vol_{\partial_{\infty}\mathbb{H}^{n+k}}(\partial_\infty \Sigma,p_0)
$$
so the conformal volume is related as the leading order behavior of this expansion.  Comparing with \cite{BWRelEnt} and \cite{BWTopUniq} it seems the renormalized area is analogous to the relative entropy in some respects which is what motivates Conjecture \ref{MainConj}

\appendix
\section{Elementary properties of $\lambda_{\mathbb{H}}$}
For the readers convenience, we collect here some elementary properties of the hyperbolic entropy.

Let $\mathbb{H}^n\subset \mathbb{H}^{n+k}$ be a totally geodesic copy of hyperbolic space.
\begin{lem}
	\label{HyperplaneEntropyLem}
	One has $\lambda_{\mathbb{H}}[\mathbb{H}^n]=1$.
\end{lem}
\begin{proof}
	One clearly has $\lambda_{\mathbb{H}}[\mathbb{H}^n]\geq 1$ so we only have to prove the reverse inequality.  To that end, by the Millison's identity \eqref{MillisonEqn} and the positivity of the heat kernel $\partial_\rho K_n= -e^{nt} \sinh(\rho)K_{n+2} \leq 0$ so $\Phi_n^{0,p_0}$ is radially decreasing. 
	Given $p_0\in \mathbb{H}^n\subset \mathbb{H}^{n+k}$ one clearly has
	$$
	\int_{\mathbb{H}^n} \Phi_n^{0,p_0}(t,p) dVol_{\mathbb{H}^n}(p)=1
	$$
	while for $p_0\not \in \mathbb{H}^n$, let $p_0'\in \mathbb{H}^n$ be the nearest point in $\mathbb{H}^n$ to $p_0$.  For $p\in \mathbb{H}^n$ it follows from the hyperbolic law of cosines that $\dist_{\mathbb{H}^{n+k}}(p,p_0')\leq \dist_{\mathbb{H}^{n+k}}(p,p_0)$ and so for all $p\in \mathbb{H}^n$
	$$
	\Phi_n^{0, p_0}(t,p )=K_n(t, \dist_{\mathbb{H}^{n+k}}(p,p_0))\leq K_n(t, \dist_{\mathbb{H}^{n+k}}(p,p_0'))=\Phi_n^{0, p_0'}(t,p)
	$$
	and the claim follows.
\end{proof}
We also verify that the hyperbolic entropy is finite on closed submanifolds
\begin{lem}
	If $\Sigma\subset \mathbb{H}^{n+k}$ is a $n$-dimensional closed submanifold, then there is a $p_0, \tau$ so
	$$\lambda_{\mathbb{H}}[\Sigma]=\int_{\Sigma} \Phi_n^{0,p_0}(-\tau,p)dVol_\Sigma(p)<\infty.
	$$
\end{lem}
\begin{proof}
	Fix a point $p_1\in \mathbb{H}^{n+k}$.  As $\Sigma$ is closed, there is an $R_0$ so $\Sigma\subset B_{R_0}^{\mathbb{H}^{n+k}}(p_1)$.  
	By the decay estimates on $K_n$ of \ref{KnDecayEst} one has for any $\epsilon>0$ an $R_\epsilon>R_0$, so if $p_0\in  \mathbb{H}^{n+k}\setminus  B_{R_0}^{\mathbb{H}^{n+k}}(p_1)$ then for any $\tau>0$ and $t_0\in \Real$
	$$
	\int_{\Sigma} \Phi_{n}^{0,p_0} (-\tau,p) dVol_{\Sigma}(p)< \epsilon.
	$$
	Similarly, by \eqref{TimeDecayEqn} for $p_0\in  B_{R_0}^{\mathbb{H}^{n+k}}(p_1)$ one has
$$
\lim_{\tau\to \infty} 	\int_{\Sigma} \Phi_{n}^{0,p_0} (-\tau,p) dVol_{\Sigma}(p)=0.
$$
As $\Sigma$ is smooth, one has, for any fixed $p_0$, that
$$
\lim_{\tau\to 0} 	\int_{\Sigma} \Phi_{n}^{0,p_0} (-\tau,p) dVol_{\Sigma}(p)\leq 1.
$$
As such, there is a $\tau_1\in (0,1)$ and an $R_1>R_0$ so
$$
\lambda_{\mathbb{H}}[\Sigma]=\sup_{p_0\in  \bar{B}_{R_1}^{\mathbb{H}^{n+k}}(p_1), \tau\in [\tau_1, \tau_1^{-1}]} \int_{\Sigma} \Phi_{n}^{0,p_0} (-\tau,p) dVol_{\Sigma}(p).
$$
As this is the supremem over a compact set of a continuous function the supremum is achieved and has a finite value.
\end{proof}
\bibliographystyle{hamsabbrv}
\bibliography{Library}
\end{document}